\documentclass[12pt]{article}
\usepackage{amsmath,amssymb,amsthm, amsfonts}
\usepackage{hyperref}
\usepackage{graphicx}
\usepackage{url}
\usepackage[mathlines]{lineno}
\usepackage{dsfont} 
\usepackage[dvipsnames]{xcolor}

\usepackage{color}
\definecolor{red}{rgb}{1,0,0}

\definecolor{blue}{rgb}{0,0,1}

\definecolor{green}{rgb}{0,.6,0}

\usepackage{float}

\usepackage{tikz}

\newtheorem{thm}{Theorem}[section]

\newtheorem{lem}[thm]{Lemma}
\newtheorem{prop}[thm]{Proposition}

\newtheorem{obs}[thm]{Observation}

\newtheorem*{main1}{Theorem 1.1}
\newtheorem*{main2}{Theorem 1.2}

\theoremstyle{definition}

\theoremstyle{definition}

\theoremstyle{definition}
\newtheorem{rem}[thm]{Remark}

\theoremstyle{definition}

\theoremstyle{definition}

\newcommand{\Z}{\mathbb{Z}}

\newcommand{\rb}{\text{rb}}

\newcommand{\bit}{\begin{itemize}}
\newcommand{\eit}{\end{itemize}}
\newcommand{\ben}{\begin{enumerate}}
\newcommand{\een}{\end{enumerate}}
\newcommand{\beq}{\begin{equation}}
\newcommand{\eeq}{\end{equation}}
\newcommand{\bea}{\begin{eqnarray}} 
\newcommand{\eea}{\end{eqnarray}}
\newcommand{\bpf}{\begin{proof}}
\newcommand{\epf}{\end{proof}\ms}
\newcommand{\bmt}{\begin{bmatrix}}
\newcommand{\emt}{\end{bmatrix}}
\newcommand{\ms}{\medskip}

\newcommand{\beqs}{\begin{equation*}} 
\newcommand{\eeqs}{\end{equation*}}
\newcommand{\beas}{\begin{eqnarray*}}
\newcommand{\eeas}{\end{eqnarray*}}

\title{Rainbow Solutions to the Sidon Equation in Cyclic Groups}

\author{Zhanar Berikkyzy\thanks{Mathematics Department, Fairfield University, Fairfield, CT, USA (zberikkyzy@fairfield.edu)}\and 
J\"urgen Kritschgau\thanks{Dept.~of Mathematics, Iowa State University, Ames, IA, USA (jkritsch@iastate.edu) Research is supported by NSF grant DMS-1839918} }
\begin{document}

\maketitle

\begin{abstract}
Given a coloring of group elements, a rainbow solution to an equation is a solution whose every element is assigned a different color.
The rainbow number of $\mathbb{Z}_n$ for an equation $eq$, denoted $rb(\mathbb{Z}_n,eq)$, is the smallest number of colors $r$ such that every exact $r$-coloring of $\mathbb{Z}_n$ admits a rainbow solution to the equation $eq$.
We prove that for every exact $4$-coloring of $\mathbb{Z}_p$, where $p\geq 3$ is prime, there exists a rainbow solution to the Sidon equation $x_1+x_2=x_3+x_4$. 
Furthermore, we determine the rainbow number of $\mathbb{Z}_n$ for the Sidon equation. 

\end{abstract}
\section{Introduction}
An {\it $r$-coloring} of a set $X$ is a function $c:X\rightarrow [r]$, where $[r]=\{1,2,\ldots,r\}$.
An $r$-coloring is {\it exact} if the function $c$ is surjective. 
In this paper we focus on $r$-colorings of the cyclic group $\Z_n$, and all $r$-colorings are assumed to be exact. 
A subset $A\subseteq \mathbb{Z}_n$ is called {\it rainbow} (under the coloring $c$) if each element of $A$ is colored distinctly. 
Given an equation $eq$, the {\it rainbow number} of $\Z_n$ for $eq$, denoted $\rb(\Z_n,eq)$, is the smallest positive integer $r$ such that every (exact) $r$-coloring of $\Z_n$ has a rainbow solution to $eq$. 
By convention, if no such integer $r$ exists we set $\rb(\Z_n,eq)=n+1$. 
We say that a coloring $c$ is \emph{rainbow $eq$ free} if there is no rainbow solution to $eq$ under $c$.

Rainbow numbers of $\Z_n$ and $[n]$ for the equation $x_1+x_2=2x_3$, for which the solutions are $3$-term arithmetic progressions, are known as anti-van der Waerden numbers. 
Rainbow arithmetic progressions have been studied extensively in \cite{AF, BSY, BEHHKKLMSWY,JLMNR,Y}. 
Generalizing the equation $x_1+x_2=2x_3$, Bevilacqua et. al. in \cite{BKKTTY} determined the rainbow number of $\Z_n$ for the equation $x_1+x_2=kx_3$ when $k=1$ (which is known as the Schur equation) or when $k$ is prime. 
Ansaldi et.al. determined the rainbow number of $\Z_p$ for the equation $a_1x_1+a_2x_2+a_3x_3=b$, and established the rainbow number of $\Z_n$ for this equation under certain conditions on the coefficients in \cite{AEHNWY}. 
Structures of rainbow-free colorings for various equations have been studied in \cite{HM, LM, MS}.

In this paper, we establish the rainbow number of $\Z_n$ for the Sidon equation $x_1+x_2=x_3+x_4$. 
The Sidon equation is a classical object in additive number theory and is used to measure the additive energy of a set (see \cite{O,TV}). 
Fox, Mahdian, and Radoi\^{c}i\'{c} showed in \cite{FMR} that for every $4$-coloring of $[n]$, where each color class has cardinality more than $\frac{n+1}{6}$, there exists a rainbow solution to the Sidon equation. 
The lower bound on a color class cardinality is tight. 
Taranchuk and Timmons in \cite{TT} studied the maximum number of rainbow solutions to the Sidon equation for a fixed number of colors.

In this paper we denote the Sidon equation by $S$. We determine $\rb(\Z_p,S)$, where $p$ is prime in Section \ref{Zp}. 
Furthermore, we determine $\rb(\Z_n,S)$ in Section \ref{Zn}. 
Notice that $\rb(\Z_2,S)=3$ and $\rb(\Z_3,S)=4$ by convention. 
Our main results are as follows: 

\begin{thm}\label{main1}
Let $p\geq 3$ be a prime. Then $rb(\Z_p,S)=4$.
\end{thm}

\begin{thm}
Let $n=p_1\cdots p_k$ be a prime factorization such that $p_i\leq p_j$ whenever $i<j$. 
Let $m$ be the smallest index such that $p_m\geq 3$ (or $m=k$ if this index does not exist),  $f_1=|\{p_i: p_i\leq 3, i\neq m\}|$, and  $f_2=|\{p_i:p_i\geq 5, i\neq m\}|$. 
Then 
\[\rb(\Z_{p_m},S)+f_1+2f_2
= \rb(\Z_n,S).\]
\end{thm}

\section{Rainbow Numbers for $\Z_p$}\label{Zp}

The section is structured as follows: First, we prove lemmas that let us make some assumptions about the structure of a rainbow Sidon free coloring of $\Z_p$.
In particular, Lemma \ref{dom} shows that $i$-dominant colors must exist for all $i\in \Z^*_p$ and Lemma \ref{transandscale} lets us translate and scale colorings, where $\Z^*_p$ is the cyclic multiplicative group. 
The remainder of the lemmas prove various structural results about rainbow Sidon free $4$-coloring on $\Z_p$. 
In fact, not all of the structural requirements of a rainbow Sidon free $4$-coloring on $\Z_p$ can be satisfied. 
In this sense, the entire section is to be read as a proof by contradiction, in which the contradiction is found in the proof of Theorem \ref{main1}. 

As an interesting note, we would like to point out that Lemmas \ref{nodouble} and \ref{2dom} have analogous results in the context of the interval $[n]$ as shown in \cite{FMR} by Fox et. al. 
In these cases, our proofs are similar, but expedited by the assumptions Lemmas \ref{dom} and \ref{transandscale} afford us.
Curiously,  Lemma \ref{nodoubley} contrasts with Lemma 3 in \cite{FMR}, where it is shown that a $YY$-string must exist.

As noted in the introduction, $\rb(\Z_2,S)=3$ and $\rb(\Z_3,S)=4$ by convention. 
The following observation is a trivial lower bound for $\rb(\Z_p,S)$ when $p\geq 5$.

\begin{obs}\label{trivial}
Let $p\geq 3$ be a prime. Then $\rb(\Z_p,S)>3$. 
\end{obs}

To prove the corresponding upper bound, we show that any $4$-coloring of $\Z_p$ with $p\geq 5$ prime admits a rainbow solution to the Sidon equation.
Suppose $c:\mathbb{Z}_p\to \{R, Y, G, B\}$ is a rainbow Sidon free $4$-coloring. We say that a color $X$ is \emph{dominant} if for any pair of elements $x,x+1\in \Z_p$, either $c(x)=c(x+1)$ or $X\in \{c(x),c(x+1)\}$. More generally, we say a color $X$ is \emph{$i$-dominant} if for any pair of elements $x,x+i\in \Z_p$, either $c(x)=c(x+i)$ or $X\in \{c(x),c(x+i)\}$. 

For a color $X$, an interval of integers $[i,i+j]$ is an \emph{$X$-string} if $c([i,i+j])=\{X\}$. 
Similarly, given two colors $X_1,X_2$, an interval $[i,i+j]$ is an \emph{$X_1X_2$-string} if $c([i,i+j])= \{X_1,X_2\}$ (these strings are also called bichromatic).
An $X$ or $X_1X_2$-string $[i,i+j]$ is \emph{maximal} if $c(i-1),c(i+j+1)\notin \{X\}$ or $c(i-1),c(i+j+1)\notin \{X_1,X_2\}$, respectively. 
A pattern (sequence) of colors $X_0X_1X_2\cdots X_k$ \emph{appears at position $j$} if $c(j+i)= X_i$ for $0\leq i\leq k$; if such a $j$ does not exist, then $X_0X_1X_2\cdots X_k$\emph{ does not appear}.
A string $A$ is \emph{$i$-periodic} if for all $x,x+i\in A$ we have $c(x)=c(x+i)$. 
Often we will abuse notation and identify a string $A$ with its induced pattern of colors.

\begin{lem}\label{dom}
Every rainbow Sidon free $4$-coloring $c$ of $\Z_p$ has an $i$-dominant color for any $i\in \Z_p\setminus\{0\}$.
\end{lem}

\begin{proof}
Let $i\in \Z_p\setminus\{0\}$. 
Form a graph $H$ on $V(H)=\{R,Y,G,B\}$ where $X_1X_2\in E(H)$ if and only if there exists $x\in \Z_p$ such that $\{c(x),c(x+i)\}=\{X_1,X_2\}$. 
Notice that $\delta(H)\geq 1$, since $\Z_p$ does not have any proper subgroups generated by $i$. 
By construction, if $H$ contains a $2K_2$ subgraph, then $c$ admits a rainbow solution to the Sidon equation. 
Therefore, $H$ is $2K_2$-free. Since $\delta(H)\geq 1$ and $H$ is $2K_2$-free, $H$ must be isomorphic to $K_{1,3}$. 
Let $X\in V(H)$ such that $d(X)=3$. 
Notice that $X$ is an  $i$-dominant color.
\end{proof}

The next lemma shows that the rainbow Sidon free property of colorings is preserved by translating and scaling colorings. 

\begin{lem}\label{transandscale}
Let $c$ be a coloring of $\Z_p$, $i\in \Z_p^*$, and $j\in \Z_p$. Let $c_{i,j}$ be given by $c_{i,j}(x)=c(ix+j)$. The coloring $c$ is rainbow Sidon free if and only if $c_{i,j}$ is rainbow Sidon free.
\end{lem} 
\begin{proof}
Let $A=\{x_1,x_2,x_3,x_4\}\subset \Z_p$. 
Notice that $A$ is rainbow under $c$ if and only if  $A_{i,j}=\{ix_1+j,ix_2+j,ix_3+j,ix_4+j\}$ is rainbow under $c_{i,j}$. 
Furthermore, $x_1+x_2=x_3+x_4$ if and only if $ix_1+j+ix_2+j=ix_3+j+ix_4+j$.
\end{proof}

It should be noted that Proposition 3.5 in \cite{BEHHKKLMSWY} and Theorem 3.5 in \cite{JLMNR} together determine when $\rb(\Z_p,eq)=3$ and when $\rb(\Z_p,eq)=4$, where $eq$ is the equation $x_1+x_2=2x_3.$
Since we use Proposition 3.5 from \cite{BEHHKKLMSWY}, we have stated it below.

\begin{prop}[3.5 in \cite{BEHHKKLMSWY}]
Let $eq$ be the equation $x_1+x_2=2x_3$. For every prime number~$p$, $3\leq \rb(\Z_p,eq)\leq 4$. 
\end{prop}

Let $x,x+i,x+2i$ be an rainbow $3$-term arithmetic progression in $\Z_p$ under the coloring~$c$. 
Without loss of generality, let $c(x+i)=R$. Notice that $R$ is dominant under the coloring~$c_{i^{-1},0}$. 
Furthermore, $R$ is not $2$-dominant given $c_{i^{-1},0}$. 
Since $c$ and $c_{i^{-1},0}$ have the same behavior with respect to rainbow Sidon solutions, we will always assume that $R$ is dominant, and $R$ is not $2$-dominant. 
In particular, we will assume that the pattern $YRB$ appears in $c$ (otherwise, we can find a rainbow $3$-term arithmetic progression, and scale/translate the coloring to put ourselves in this position). 
Furthermore, since a $2$-dominant color must exist (and it is either $Y$ or $B$), we can assume that $Y$ is $2$-dominant. From this point forward in this section, we will use these assumptions to prove structural results about $c$. 
Ultimately, these structures will lead to a contradiction,in the sense that we will find a rainbow solution to the Sidon equation in the proof of Theorem \ref{main1}.

\begin{lem}\label{nodouble}
Let $X\in\{B,G\}$. Every maximal $RX$-string  is $2$-periodic with exactly one element colored by $X$ within every substring of length $2$. In particular, $BB$ and $GG$ do not appear.
\end{lem}

\begin{proof}
Let $A$ be a maximal $RG$-string. 
Recall that we assume the pattern $YRB$ appears under the coloring $c$.
Therefore, if $A$ contains patterns $RRG$, $GRR$, $RGG$, or $GGR$, then $c$ admits a rainbow solution to the Sidon equation. 
Thus, $A$ must be of the form $RGR\cdots RGR$. 

Now let $A$ be a maximal $RB$-string. Since $Y$ is $2$-dominant, we know that patter $GRY$ must appear under $c$.
This implies that $A$ cannot contain $RRB, BRR, BBR, RBB$. 
Thus, $A$ must be of the form $RBR\cdots RBR$.
\end{proof}

The next lemma extends Lemma \ref{nodouble} to shows that the pattern $YY$ cannot appear under the coloring $c$. It will be very useful since it implies that if $c(x)=Y$, then $c(x-1)=c(x+1)=R$. 

\begin{lem}\label{nodoubley}
The pattern $YY$ does not appear under the coloring $c$.
\end{lem}

\begin{proof}
Let  $i$ and $i+m$ be colored $G$ and $B$ by $c$. 
Suppose the pattern $RYY$ exists at position $j$, and consider colors of  $j+m, j+m+1,j+m+2$. 
There are three cases: $c(j+m)$ in $\{Y\}$, $\{B,G\}$, or $\{R\}$. 
If $c(j+m)=Y$, then $\{i, j+m, i+m, j\}$ is a rainbow Sidon solution. 
If $c(j+m)\in\{B, G\}$, then  $c(j+m+1)=R$ by Lemma \ref{nodouble} and  $\{i,j+m+1,i+m,j+1\}$ is a rainbow Sidon solution.
Therefore, $c(j+m)=R$. 

Next consider the color of $j+m+2$. 
There are three case: $c(j+m+2)$ is in $\{R\}$, $\{B,G\}$, or $\{Y\}$. 
If $c(j+m+2)=R$, then $\{i,j+m+2,i+m,j+2\}$ is a rainbow Sidon solution.
If $c(j+m+2)\in\{B, G\}$, then $c(j+m+1)=R$ by Lemma \ref{nodouble} and $\{i,j+m+1,i+m,j+1\}$ a is rainbow Sidon solution. 
Therefore, $c(j+m+2)=Y$. 

It follows that $c(j+m+1)\not\in\{ B,G\}$ since $c(j+m+2)=Y$ and $R$ is dominant. 
Furthermore, $c(j+m+1)\neq R$, else $\{i,j+m+1,i+m,j+1\}$ is a rainbow Sidon solution. 
Therefore, $c(j+m+1)=Y$.

Thus if there is an $RYY$-string starting at $j$, then there is an $RYY$ at $j+m$. 
Since  $m$ is a generator of $\Z_p$, Case 3.3 implies that $c(x)=R$ for all $x\in \Z_p$, which is a contradiction. 
\end{proof}

The following lemma proves that any pair of elements with colors $B$ and $G$ must have an $R$-string in the middle between them.

\begin{lem}\label{midpoint}
Suppose that $X_1RY$ appears at $i$ and $YRX_2$ appears at $i+j-2$ where $j$ is odd and $\{X_1,X_2\}=\{B,G\}$. Then $c(i+\frac{j-1}{2})=c(i+\frac{j+1}{2})=R.$
\end{lem}

\begin{proof}
Let $y=i+\frac{j-1}{2}$. 
For the sake of contradiction, suppose that $\{c(y),c(y+1)\}\neq\{R\}$. 
By the symmetry of the problem, there are three cases: $Y\in \{c(y),c(y+1)\}$, $c(y)=X_1$, and $c(y)=X_2$. 

If $c(y)=Y$, then
by the dominance of $R$ and Lemma \ref{nodoubley} $c(y+1)=R$, hence  $\{i,i+j,y,y+1\}$ is a rainbow Sidon solution.
If $c(y)=X_1$, then $c(y+1)=R$ and $c(y+2)\in \{X_1,Y\}$. 
Depending on the value of $c(y+2)$, either $\{i+1,i+j,y,y+2\}$ or $\{i+2,i+j,y+1,y+2\}$ is a rainbow Sidon solution.
If $c(y)=X_2$, then  $c(y-1)=R$ and $\{i,i+j-2,y,y-1\}$ is a rainbow  Sidon solution.

In each case, we get a contradiction. 
\end{proof}

\begin{lem}\label{2dom}
The patterns $BRB$ and $GRG$ do not appear under the coloring $c$.
\end{lem}

\begin{proof}
Without loss of generality, suppose that $BRB$ exists. 
Let $i_b,\dots, i_b+j_b$ be a maximal $BR$-string that starts and ends with $B$ (we are truncating the $R$ colored element at the start and end of a maximal $BR$-string). 
Note that this string alternates colors between $B$ and $R$.
This implies that $j_b\geq 2$ and $j_b$ is even. 
Let $i_g,\dots, i_g+j_g$ be a maximal $GR$-string that starts and ends with $G$. 
This implies that $j_g$ is even. 

Since $p$ is odd, either $i_g-i_b-j_b$ or $i_b-i_g-j_g$ is odd. 
Without loss of generality, suppose that $j=i_g-i_b-j_b$ is odd. 
Let $i=i_b+j_b$ so that $i+j=i_g$. 
In particular, we have chosen $i$  and $j$ such that $BRBRY$ appears at $i-2$ and $YRG$ appears at $i+j-2$, where $j$ is odd.

Let $y=i+\frac{j-1}{2}$. 
Notice that $c(y)=c(y+1)=R$ by Lemma \ref{midpoint}. Let $k$ be the smallest integer such that either $c(y-k)=Y$ or $c(y+1+k)=Y$. Notice that if $\{c(y-k),c(y+1+k)\}=\{R,Y\}$, then $\{i, i+j,y-k,y+1+k\}$ is a rainbow Sidon solution. Therefore, $c(y-k)=c(y+1+k)=Y$.
Notice that $c(i-2)=B$ and $c(y+1+k-2)=R$. 
However,
\[(y-k)+(y+1+k-2)=2i+j-2 =(i-2)+(i+j).\]
Therefore, $\{i-2,i+j,y-k,y+1+k-2\}$ is a rainbow Sidon solution, a contradiction. Therefore, neither $BRB$ nor $GRG$ can appear under the coloring $c$.
\end{proof}

\begin{lem}\label{1or3}
All $R$-strings have length $1$ or $3$.
\end{lem}

\begin{proof}
Let elements $i$ and $i+m$ be colored with $B$ and $G$, respectively. 
Since $m\geq 1$ is a generator of $\Z_p$, either (1) there exists $YRR$ at some position $j$ such that $YRR$ does not appear at $j+m$, or (2) every $R$-string has length $1$ and the lemma is proven. 
We now assume (1) holds, and will show that $YRRRY$ appears at position $j$. 
We will proceed by considering the colors of $j+m, j+m+1,j+m+2$. 

First we will conclude that $c(j+m)=Y$ by ruling out other three options.
If $c(j+m)=R$, then $\{i, j+m, i+m, j\}$ is a rainbow Sidon solution. If $c(j+m)=B$, then $\{i+2,j+m,i+m,j+2\}$ is a rainbow Sidon solution because $c(i+2)=Y$ by Lemma \ref{2dom}.
If $c(j+m)=G$, then $\{i,j+m,i+m-2,j+2\}$ is a rainbow Sidon solution because $c(i+m-2)=Y$ by Lemma \ref{2dom}.
Therefore, $c(j+m)=Y$.
Furthermore, by Lemma \ref{nodoubley},  $c(j+m+1)=R$ since $R$ is dominant. 

Next consider the color of $j+m+2$. 
Notice that if $c(j+m+2)=Y$, then $\{i,j+m+2,i+m,j+2\}$ is a rainbow Sidon solution.
Furthermore, recall that $j$ was selected so that $YRR$ does not appear at $j+m$. 
Therefore, $c(j+m+2)\neq R$.
In particular, $c(j+m+2)\in\{B, G\}$. 

Now $c(j+4)\neq R$, else either $\{i,j+m+2,i+m-2,j+4\}$ or $\{i+2,j+m+2,i+m,j+4\}$ is a rainbow Sidon solution depending on the color of $j+m+2$. 
Therefore, $c(j+4)=Y$ and $c(j+3)=R$. 

Thus, we have shown that if there exists a $YRR$ at position $x$, then there exists a $k\geq 1$ such that $YRX$ with $X\in\{B,G\}$ appears at position $x+km$, $YRRRY$ appears at position $x+(k-1)m$ and $YRR$ appears at $x+\ell m$ for $0\leq \ell <k$. We will show that this behavior propagates backwards from $j$. 
In particular, suppose that $YRRRY$ appears at position $j$, and that $YRR$ appears at $j-m$.
Notice that $c(j-m+3), c(j-m+4)\in\{R,Y\}$ since $Y$ is $2$-dominant.
If $c(j-m+4)=R$, then $\{i+m, j-m+4, i, j+4\}$ is a rainbow Sidon solution.
Therefore, $c(j-m+4)=Y$, and so by Lemma \ref{nodoubley} we have $c(j-m+3)=R$. Hence $YRRRY$ appears at position $j-m$.

In summary, if there exists a $YRR$ at position $x$, then there exists a $k\geq 1$ such that $YRX$ with $X\in\{B,G\}$ appears at position $x+km$, $YRRRY$ appears at position $x+(k-1)m$ and $YRR$ appears at $x+\ell m$ for $0\leq \ell <k$. 
By the argument in the previous paragraph,  $YRRRY$ appears at $x+\ell m$ for every $0\leq \ell<k$. 
Thus, every $R$-string has length $1$ or $3$.
\end{proof}

With no further ado:

\begin{main1}
The coloring $c$ does not exist. In particular, $\rb(\Z_p,S)=4$ for all prime $p\geq 3$.
\end{main1}

\begin{proof}
Since $c$ is a $4$-coloring, there must exist $i$ and $j$ such that $c(i)=B$ and $c(j)= G$.
Furthermore, since $p$ is odd, there are distinct $m_1,m_2\in \Z_p$ such that $i+m_1=j$ and $i-m_2=j$ where either $m_1$ is odd or $m_2$ is odd.
Without loss of generality, suppose that $m_1$ is odd.
By Lemma \ref{midpoint}, let $y=i+\frac{j-1}{2}$ and $c(y)=c(y+1)=R$.  
Notice that Lemma \ref{1or3} guarantees that the $R$-string containing $y$ has length $3$. 
Therefore, either $c(y-1)=Y$ and $c(y+2)=R$ or $c(y-1)=R$ and $c(y+2)=Y$.
In either case, $\{i,y-1,y+2,j\}$ is a rainbow Sidon solution. This contradicts the assumption of the section, that $c$ is a rainbow Sidon free $4$-coloring of $\Z_p$. 
\end{proof}

\section{Rainbow Numbers  for $\rb(\mathbb{Z}_n)$}\label{Zn}

This section is broken into two subsections. The lower bound for $\rb(\Z_n,S)$ is shown in the first subsection, which will provide insight into where (and how) to look for rainbow solutions to the Sidon equation in the upper bound argument. The upper bound is shown in the second subsection. 

\subsection{Lower Bound}

We construct a lower bound coloring for $\Z_n$ by expanding a coloring for $\Z_{n/p}$. 
Essentially, we insert $p-1$ elements between each pair of elements in $\Z_{n/p}$ (taken in their natural cyclic ordering), and color the elements appropriately. 
The method for coloring the new inserted elements is specifically chosen to maintain $i$-dominant colors. 

\begin{lem}\label{lbcon}
Let $c$ be a rainbow Sidon free $r$-coloring of $\Z_n$ and $p$ be prime. If $p\leq 3$, then there exists a rainbow Sidon free $(r+1)$-coloring of $Z_{pn}$. If $p\geq 5$, then there exists a rainbow Sidon free $(r+2)$-coloring of $\Z_{pn}$. 
\end{lem}

\begin{proof}
Assume that $p\leq 3$. 
Let 
\[\hat c(x)=\begin{cases}c(x/p)& x\equiv 0 \mod p\\
r+1& \text{otherwise}
\end{cases}\] be an $(r+1)$-coloring of $\Z_{pn}$. 
Let $X=\{x_1,x_2,x_3,x_4\}\subseteq  \Z_{pn}$ such that $x_1+x_2=x_3+x_4$.

Without loss of generality, if $x_1,x_2,x_3\in \langle p\rangle$, then $x_4\in \langle p\rangle$ where $\langle p\rangle$ is the subgroup  of $\Z_{pn}$ generated by $p$. 
In this case, $X$ is not rainbow, since $c$ is a rainbow free coloring of $\langle p\rangle\cong \Z_n$. 
Furthermore, if $x_i,x_j\notin \langle p\rangle$ for  $1\leq i<j\leq 4$, then $\hat c(x_i)=\hat c(x_j)$ and $X$ is not rainbow. Thus, $\hat c$ is a rainbow Sidon free $(r+1)$-coloring of $\Z_{pn}$.

Assume that $p\geq 5$. 
Let 
\[\hat c(x)=\begin{cases}
c(x/p)& x\equiv 0 \mod p\\
r+1& x\equiv 1, p-1\mod p\\
r+2&\text{otherwise}
\end{cases}
\] be an $(r+2)$-coloring of $\Z_{pn}$. 
Let $X=\{x_1,x_2,x_3,x_4\}\subseteq \Z_{pn}$ such that $x_1+x_2=x_3+x_4.$ 

As in the previous case, $c$ is a rainbow free coloring of $\langle p\rangle\cong \Z_n$. Therefore, if any three elements in $X$ are in $\langle p\rangle$, then $X$ is not rainbow. Thus, assume that exactly two elements in $X$ are in $\langle p\rangle$. In particular, assume that $x_i,x_j\in \langle p\rangle$.

Without loss of generality, if $i=1$ and $j=2$, then $x_3+x_4= s_1p$ for some integer $s_1$. Therefore, $x_3\equiv -x_4\mod p$ and $\hat c(x_3)=\hat c(x_4)$.

Without loss of generality, if $i=1$ and $j=3$, then $x_2\equiv x_4\mod p$. 
Therefore, $\hat c(x_2)=\hat c(x_4)$. 

In either case, $X$ is not rainbow. Thus, $\hat c$ is a rainbow Sidon free $(r+2)$-coloring of $\Z_{pn}$.
\end{proof}

Repeatedly applying Lemma \ref{lbcon} gives a lower bound for $\rb(\Z_n,S)$.
Notice that the statement of Proposition \ref{lb1} selects $p_m$ to maximize the number of colors used in the construction.

\begin{prop}\label{lb1}
Let $n=p_1\cdots p_k$ be a prime factorization such that $p_i\leq p_j$ whenever $i<j$. Let $m$ be the smallest index such that $p_m\geq 3$ (or $m=k$ if this index does not exist),  $f_1=|\{p_i: p_i\leq 3, i\neq m\}|$, and  $f_2=|\{p_i:p_i\geq 5, i\neq m\}|$. Then 
\[\rb(\Z_{p_m},S)+f_1+2f_2
\leq \rb(\Z_n,S).\]
\end{prop}

\begin{proof}  By starting with a rainbow free $(\rb(\Z_{p_m},S)-1)$-coloring of $\Z_{p_m}$, we can construct a rainbow free $r$-coloring  with 
\[r=\rb(\Z_{p_m},S)-1+f_1+2f_2\]
by repeatedly applying Lemma \ref{lbcon} for primes $p_i$, $i\neq m$.
\end{proof}

\subsection{Upper bound}

Since the lower bound construction expands colorings depending on primes $p$, it is intuitive that an upper bound argument would reduce colorings modulo $p$ until a rainbow Sidon solution can be found.
Let $t$ be a be a positive integer that divides $n$. 
Let $R_i=i+\langle t\rangle$ be the $i^{th}$ coset of the subgroup generated by $t$ in $\Z_n$.
This is notation that is consistent with work in \cite{BKKTTY}. 
Lemma \ref{maximumcoset} identifies which coset of $t$ we want to narrow in on. 

\begin{lem}\label{maximumcoset}
Let $t=n/p$ ($p$ prime) divide $n$ and $c$ be a rainbow Sidon free coloring of $\Z_n$. Consider the cosets of $\langle t\rangle$,  $R_i$ with $0\leq i < t$. There exists an index $j$ such that $|c(R_i)\setminus c(R_j)|\leq 1$ for all $0\leq i<t$. 
\end{lem}

\begin{proof}
Let $j$ be the index that maximizes $|c(R_j)|$. 
For the sake of contradiction, assume that $|c(R_i)\setminus c(R_j)|\geq 2$ for index $i$. 
This implies that there exists $x_1,x_2\in R_i$ such that $c(x_1)\neq c(x_2)$ and $c(x_1),c(x_2)\notin c(R_j)$. 
Let $x_3\in R_j$. 
Notice that 
\begin{align*}
    x_1&=s_1t+i\\
    x_2&=s_2t+i\\
    x_3&=s_3t+j.
\end{align*}
Therefore, \[x_4=x_1+x_3-x_2=(s_1-s_2+s_3)t+j.\]
Since $c$ is rainbow Sidon free, $c(x_4)=c(x_3)$. Notice that $x_4-x_3= t(s_1-s_2)$ and that $|R_j|=p$. 
Since $s_1-s_2\neq 0$, and $s_1-s_2$ is relatively prime to $p$, we can conclude that $|c(R_j)|=1$. 
This contradicts our choice of index $j$. 
\end{proof}

The next lemma focuses our attention on the coset with the most number of colors under a coloring $c$. Furthermore, it controls the number of colors lost in the process.
 
\begin{lem}\label{ublemma}
Let $p$ be a prime divisor of $n$ and $pt=n$. Then 
\[ \rb(\Z_n,S)\leq \rb(\Z_p,S)+\rb(\Z_t,S)-2.\]
\end{lem}

\begin{proof}
Let $c$ be a rainbow Sidon free $(\rb(\Z_n,S)-1)$-coloring of $\Z_n$. 
Let $R_i$ be the $i^{th}$ coset of $\langle t\rangle$ in $\Z_n$. 
By Lemma \ref{maximumcoset}, there exists index $j$ such that $|c(R_i)\setminus c(R_j)|\leq 1$ for $0\leq i< t$. 
Notice that $c$ is rainbow Sidon free if and only if $c'$ given by $c'(x)=c(x+j)$ is rainbow Sidon free. 
Since $R_0\cong \Z_p$ is a subgroup of $\Z_n$, it follows $c'$ must be a rainbow Sidon free coloring of $\Z_p$. Furthermore, $|c'(R_i)\setminus c'(R_0)|\leq 1$ for $0\leq i <t$. 

Let 
\[\hat c(x)=\begin{cases}
i& \{i\}=c'(R_x)\setminus c'(R_0)\\
\alpha & c'(R_x)\subseteq c'(R_0)\end{cases}\] be a coloring of $\Z_t$ (so that $\alpha$ is a color not used by $c'$).
For the sake of contradicion, let $\{x_1,x_2,x_3,x_4\}\subseteq \Z_t$ be rainbow given $\hat c $ such that $x_1+x_2=x_3+x_4$. 
Without loss of generality, assume that $\hat c(x_1),\hat c(x_2),\hat c(x_3)\neq \alpha$. 
Therefore, there exists $y_i\in R_{x_i}$ such that $c'(y_i)=\hat c(x_i)$ for $1\leq i \leq 3$. 
Notice that $y_4=y_1+y_2-y_3\in R_{x_4}$ and that $c'(y_4)\neq c'(y_1),c'(y_2),c'(y_3)$. 
Therefore, $\{y_1,y_2,y_3,y_4\}$ is a rainbow Sidon solution in $\Z_n$ given $c'$; this is a contradiction. 
Thus, $\hat c$ is a rainbow Sidon free coloring of $\Z_t$. 

We can combine all this information to bound the number of colors of used by $c$. In particular,
\[c(\Z_n)=c'(\Z_n)= (c'(R_0)\cup \hat c(\Z_t))\setminus \{\alpha\}.\] 
This implies that 
\[\rb(\Z_n,S)-1=|c(\Z_n)|=|c'(R_0)|+|\hat{c}(\Z_t)|-1\leq  \rb(\Z_p,S)-1 +\rb(\Z_t,S)-2,\]
completing the proof.
\end{proof}

Proposition \ref{ub1} is the result of repeatedly applying Lemma \ref{ublemma}. 
In this sense, it reverses the process used in the construction of a rainbow Sidon free coloring in the proof of Proposition \ref{lb1}. 
It is helpful to recall that $\rb(\Z_p,S)=4$ for all primes $p\geq 3$ while comparing these two propositions.

\begin{prop}\label{ub1}
Let $n=p_1\cdots p_k$ be a prime factorization of $n$. Then 
\[\rb(\Z_n,S)\leq 2(1-k)+\sum_{i=1}^{k} \rb(\Z_{p_i},S).\]
\end{prop}

\begin{proof}
Recursively apply Lemma \ref{ublemma}.
\end{proof}

Notice that the upper bound given in Proposition \ref{ub1} does not meet the lower bound in Proposition \ref{lb1}.
In particular, if $3^k|n$ for $k\geq 2$, then the upper bound exceeds the lower bound by at least $k-1$. 
This suggests that Lemma \ref{ublemma} is too generous when $p=3$.

To improve the upper bound, we want to focus on the case when $n=3t.$ 
Suppose that $c$ is an $r$-coloring of $\Z_n$. 
Since the goal is to show that $c$ admits a rainbow solution to the Sidon equation, we will assume that $c$ is rainbow Sidon free and pursue a contradiction. 
Partition $\Z_n$ into cosets of $\langle t\rangle$ denoted $R_i$, $1\leq i\leq t$.
By Lemma \ref{maximumcoset}, there exists an index $j$ such that $|c(R_i)\setminus c(R_j)|\leq 1$ for all $0\leq i <t$. By shifting the coloring, we can assume that $j=0$. 

\begin{lem}\label{monolem}
If $|c(R_i)\setminus c(R_0)|=1$ and $|c(R_0)|=3$, then $|c(R_i)|=1$. 
\end{lem}

\begin{proof}
Notice that $R_0=\{0,t,2t\}$, while $R_i=\{i,t+i,2t+i\}$.
Without loss of generality, suppose that $c(st+i)\notin c(R_0)$ for some $0\leq s\leq 2$.
For the sake of contradiction, suppose that $c((s\pm 1)t+i)\neq c(st+i)$ (where $(s\pm 1)t+i$ is taken modulo $n$).
By assumption, there exists $s't,(s'+1)t\in R_0 $ such that $c((s\pm 1)t+i)\notin \{ c(s't),c((s'\pm 1)t)\}.$ 
However, $(s'\pm 1)t+st+i=s't+(s\pm 1)t+i$. 
Since $\{s't,(s\pm 1)t+i,(s'\pm 1)t,st+i\}$ is rainbow under~$c$, a contradiction.
\end{proof} 

Rainbow numbers for the Schur equation $x_1+x_2=x_3$ will be useful in analyzing the rainbow number for the Sidon equation when $9$ divides $n$. 
For convenience, we state the relevant results below. 
Let $\rb(\Z_n,1)$ denote the fewest number of colors that guarantee a rainbow solution to $x_1+x_2=x_3$ in $\Z_n$. 

\begin{thm}[Theorem 1 in \cite{BKKTTY}]\label{schurprimes}
For a prime $p\geq 5$, $\rb(\Z_p,1)=4$. 
\end{thm}

\begin{rem}
It can be deduced through inspection that $\rb(\Z_2,1)=\rb(\Z_3,1)=3$. 
\end{rem}

This result is important for us because $\rb(\Z_p,S)=\rb(\Z_p,1)$ except when $p=3$. 
In the case that $p=3$, $\rb(\Z_p,1)=\rb(\Z_p,S)-1$. 

\begin{thm}[Theorem 2 in \cite{BKKTTY}]\label{schurresult}
For a positive integer $n$ with prime factorization $n=p_1\cdot p_2\cdots p_k$,
\[\rb(\Z_n,1)=2(1-k)+\sum_{i=1}^k\rb(\Z_{p_i},1).\]
\end{thm} 

The following fact is immediate from Theorems \ref{schurprimes} and \ref{schurresult}, and Proposition \ref{lb1}:

\begin{obs}\label{schurobs}
If $3$ divides $n$, then $rb(\Z_n,S)\geq 1+\rb(\Z_n,1)$.
\end{obs}

We use this fact to find solutions to $x_1+x_2=x_3$ in the proof of Lemma \ref{ublemma2}. 
To our knowledge, this is the first time in the literature on rainbow solutions to equations in $\Z_n$ where previously known rainbow numbers  for different equation are employed in the proof.
This method may be useful in proving rainbow numbers for more general $4$-term equations.

\begin{lem}\label{ublemma2}
Let $9$ be a divisor of $n$ and $3t=n$. Then 
\[ \rb(\Z_n,S)\leq \rb(\Z_3,S)+\rb(\Z_t,S)-3.\]
\end{lem}

\begin{proof}
Suppose that $c$ is a rainbow Sidon free $r$-coloring of $\Z_n$ where $r=\rb(\Z_3,S)+\rb(\Z_t,S)-3$. 
We choose not not evaluate $\rb(\Z_3,S)$ to $4$ for conceptual clarity. 
It is easier to keep track of and compare the number of colors used in the un-evaluated form.
Without loss of generality, we assume that $|c(R_i)\setminus c(R_0)|\leq 1$ and that $|c(R_0)|\geq |c(R_i)|$ for all $i$ (otherwise, we can shift the coloring to put ourselves in this position).
Let 
\[\hat c(x)=\begin{cases}
i& \{i\}=c(R_x)\setminus c(R_0)\\
\alpha & c(R_x)\subseteq c(R_0)\end{cases}\] be a coloring of $\Z_t$ (so that $\alpha$ is a color not used by $c$).
Notice that $\hat c$ is a rainbow Sidon free $(\rb(\Z_t,S)-1)$-coloring of $\Z_t$. 
This implies that $|c(R_0)|=3$.
By Observation \ref{schurobs}, $\hat c$ admits a rainbow solution to $x_1+x_2=x_3$ in $\Z_t$. 
In particular, suppose that $i+j=k$ such that $\{i,j,k\}$ is rainbow under $\hat c$. 
Without loss of generality, $\hat c(i)\neq \alpha$. 
By construction, if $\hat c(j),\hat c(k)\neq \alpha$, then there exists $s_jt+j$ and $s_kt+k$ such that $c(s_jt+j)=\hat c(j)$ and $c(s_kt+k)=\hat c(k)$.
If either $\hat c(j)=\alpha$ or $\hat c(k)=\alpha$, then let $s_jt+j$ (resp. $s_kt+k$) be some element in $R_j$ (resp. $R_k$).
Furthermore, there exists $s_0t\in R_0$ such that $c(s_0t)\notin \{c(s_jt+j), c(s_kt+k),\hat c(i)\}$. 
By construction, there exists $s_i$ such that \[s_kt+k+s_0t-s_jt-j=s_it+i\in R_i.\] 
Furthermore, $|c(R_i)|=1$ by Lemma \ref{monolem}.
In particular, $c(s_it+i)=\hat c(i)$ and $\{s_it+i,s_jt+j,s_kt+k,s_0t\}$ is a rainbow solution to the Sidon equation under $c$. 
This is a contradiction; therefore, any $r$-coloring of $\Z_n$ is not rainbow Sidon free. 
\end{proof}

Notice that Lemma \ref{ublemma2} is a refinement of Lemma \ref{ublemma}. 
This refinement lets us improve the upper bound in Proposition \ref{ub1} enough to meet the lower bound in Proposition \ref{lb1}.

\begin{main2}
Let $n=p_1\cdots p_k$ be a prime factorization such that $p_i\leq p_j$ whenever $i<j$. 
Let $m$ be the smallest index such that $p_m\geq 3$ (or $m=k$ if this index does not exist),  $f_1=|\{p_i: p_i\leq 3, i\neq m\}|$, and  $f_2=|\{p_i:p_i\geq 5, i\neq m\}|$. 
Then 
\[\rb(\Z_{p_m},S)+f_1+2f_2
= \rb(\Z_n,S).\]
\end{main2}

\begin{proof}
The lower bound is provided by Proposition \ref{lb1}. 
Notice that if $9$ does not divide $n$, then the upper bound in Proposition \ref{ub1} meets the lower bound. 
Therefore, assume that $9$ divides $n$.

Let $\alpha$ be the largest integer such that $3^{\alpha}$ divides $n$. 
To prove the upper bound, iterativly apply Lemma \ref{ublemma2} $\alpha-1$ times to conclude that 
\[\rb(\Z_n,S)\leq (\alpha-1)(\rb(\Z_3,S)-3)+\rb(\Z_{n/3^{\alpha-1}},S).\]
By Proposition \ref{ub1}, 
\[\rb(\Z_n,S)\leq  (\alpha-1)(\rb(\Z_3,S)-3)+\rb(\Z_m,S)+2(-k+\alpha)+\sum_{\substack{p_i\neq 3\\i\neq m}}\rb(\Z_{p_i},S).\]
Let $\beta$ be the largest integer such that $2^\beta$ divides $n$. 
Notice that $\alpha+\beta-1= f_1.$
By regrouping terms and evaluating $\rb(\Z_{p_i},S)$, 
\begin{align*}
    \rb(\Z_n,S)&\leq \rb(\Z_m,S)+(\alpha-1)(\rb(\Z_3,S)-3)+\beta(\rb(\Z_2,S)-2)+ \sum_{\substack{p_i\neq 2,3\\i\neq m}}(\rb(\Z_{p_i},S)-2)\\
    &=\rb(\Z_m,S)+f_1+2f_2.
\end{align*}
This concludes the proof.
\end{proof}

\subsection*{Acknowledgements}
The second author would like to thank Michael Young for his mentorship and helpful discussions while working on this project.
This material is based upon work supported by the National Science Foundation under Grant Number DMS-1839918.

\end{document}